\documentclass[12pt]{article}

\usepackage{geometry}
\usepackage{amssymb}
\usepackage{amsmath}
\usepackage{amsthm}
\usepackage{amsopn}
\usepackage{amscd}
\usepackage{exscale}
\usepackage{graphicx}
\usepackage{amssymb}
\usepackage{cite}
\usepackage[all,cmtip]{xy}
\usepackage{adjustbox}

\theoremstyle{plain}
\newtheorem{thm}{Theorem}[section]
\newtheorem{prop}[thm]{Proposition}

\newtheorem{cor}[thm]{Corollary}

\newtheorem{lem}[thm]{Lemma}

\newtheorem{conj}[thm]{Conjecture}

\theoremstyle{definition}

\theoremstyle{remark}

\newtheorem*{thma}{{\bf Theorem A}}
\newtheorem*{examplea}{{\bf Example A}}
\newtheorem*{propa}{{\bf Proposition A}}

\newcommand{\C}{\mathcal{C}}

\newcommand{\W}{{\mathcal{W}}}
\newcommand{\CC}{{\mathbb{C}}}
\newcommand{\OO}{{\mathcal{O}}}
\newcommand{\Q}{\mathbb{Q}}
\newcommand{\LL}{\mathcal{L}}
\newcommand{\R}{\mathbb{R}}
\newcommand{\Z}{\mathbb{Z}}
\newcommand{\HH}{\mathcal{H}}

\newcommand{\TT}{\mathcal{T}}

\author{Dipramit Majumdar}
\title{Geometry of the eigencurve at critical Eisenstein series of weight 2 }
\date{}
\begin{document}
\maketitle

\begin{abstract}
In this paper we show that the critical Eisenstein series of weight 2, $E_{2}^{crit_{p}}$, is smooth in the eigencurve $\C(l)$, where $l$ is a prime. We also show that $E_{2}^{crit_{p},ord_{l}}$ is smooth in the full eigencurve $\C^{full}(l)$ and $E_{2}^{crit_{p},ord_{l_{1}},ord_{l_{2}}}$ is non-smooth in the full eigencurve $\C^{full}(l_{1}l_{2})$. Further, we show that, $E_{2}^{crit_{p}}$, is \'etale over the weight space in the eigencurve $\C(l)$. As a consequence, we show that level lowering conjecture of Paulin fails to hold at $E_{2}^{crit_{p},ord_{l}}$. 
\end{abstract}
\tableofcontents

\section{Introduction}
\label{intro}
Coleman and Mazur \cite{Coleman.1998} introduced a rigid analytic curve $\C$ parametrizing finite slope overconvergent $p$-adic eigenforms of tame level $1$, called eigencurve. This was axiomatized and generalized by Buzzard \cite{Buzzard.2007} to all levels. Buzzard's eigenvariety machine feeds in a family of $p$-adic Banach spaces\;(like space of overconvergent $p$-adic modular forms), and gets out geometric object called eigenvariety.\\
 
It was known due to work of Hida \cite{MR848685} that ordinary classical points of weight greater or equal to two in the eigencurve are smooth and  \'etale over the weight space. Coleman and Mazur \cite{Coleman.1998} showed for tame level 1, and  Kisin \cite{MR1992017} for arbitrary tame level eigencurve, that the non-critical classical points are smooth and \'etale over the weight space, provided $\alpha \neq \beta$\;(which is conjectured and known for $k=2$). For $k>2$, Bella\"\i che and Chenevier \cite{Bellaiche.2006b} proved that the critical Eisenstein series $E_{k}^{crit_{p}}$ is a smooth point which is  \'etale over the weight space.  On the other hand, for the weight 1 case, it is shown by the work of Bella\"\i che and Dimitrov  \cite{Joel-Dimitrov}, that at regular points the eigencurve is smooth over the weight space. By the work of Dimitrov and Ghate \cite{Dimitrov-Ghate} it is known that non-regular weight 1 modular forms can be non-smooth over the weight space. Another possible example of non-smooth classical point of weight greater or equal to two is the critical Eisenstein series of weight two. In this paper, we study the geometry of the eigencurve at point corresponding to critical Eisenstein series of weight $2$. For a brief overview of eigencurve and critical Eisenstein series of weight $2$, we refer reader to the beginning of section $2$.\\

 \begin{thma}
Let $\ell \neq p$ a prime. Eigencurve of tame level $\ell$, $\C(\ell)$ is smooth and \'etale over the weight space at $E_{2}^{crit_{p}}$.\\
\end{thma}
This is a combination of Theorem \ref{E2smatl} and Theorem \ref{E2etale}. Our approach to prove $\C(\ell)$ is smooth at $E_{2}^{crit_{p}}$ is similar to the method of Bellaiche and Chenevier as in \cite{Bellaiche.2006b}, but in this situation we have a strict inclusion of selmer group $H_{f}^{1} \subset H^1$ and one has to verify that the extension coming from Ribet's lemma is crystalline. In order to show that in $\C(\ell)$, $E_{2}^{crit_{p}}$ is \'etale over the weight space, using tree structure of $GL_{2}$, we prove that the reducibility locus at $p$ of the pseudo-character $\TT|_{G_{\Q_{p}}}$ is the maximal ideal.\\

The paper will also prove a statement regarding smoothness, analogous to Theorem A for a nontrivial Dirichlet character $\chi$ of prime conductor. We also prove the following result regarding the smoothness of $E_{2}^{crit_{p},ord_{\ell}}$ in the full eigencurve and full cuspidal eigencurve.

\begin{propa} 
The point $E_{2}^{crit_{p},ord_{\ell}}$ is smooth in the full eigencurve of tame level $\ell$, $\C^{full}(\ell)$ and the full cuspidal eigencurve of tame level $\ell$, $\C^{0,full}(\ell)$.\\
\end{propa}

This is Corollary \ref{E2fullsmatl}. We show that the map $f: \C^{full}(\ell) \to \C(\ell)$ is locally isomorphism at point $E_{2}^{crit_{p}}$. Since $E_{2}^{crit_{p}}$ is smooth in $\C(\ell)$, the result follows. As a consequence (Corollary \ref{nonsmooth}), we find an example of non-smooth classical point of weight $2$ in the full eigencurve $\C^{full}(\ell_{1}\ell_{2})$. 

\begin{examplea}
Let $\ell_{1},\ell_{2} \neq p$ be two distinct primes. The point corresponding to $E_{2}^{crit_{p},ord_{\ell_{1}},ord_{\ell_{2}}}$ is non-smooth in the full eigencurve $\C^{full}(\ell_{1}\ell_{2})$ of tame level $\ell_{1}\ell_{2}$.
\end{examplea}

Paulin made a conjecture regarding non-smooth points in the full cuspidal eigencurve \cite{Paulin.2012}. Level lowering conjecture of Paulin predicts that there are two components of $\C^{0,full}(\ell)$, one generically special and one generically principal series passing through the point $E_{2}^{crit_{p},ord_{\ell}}$. Since the eigencurve $\C^{0,full}(\ell)$, is smooth at $E_{2}^{crit_{p},ord_{\ell}}$, we see that the level lowering conjecture of Paulin fails to hold at $E_{2}^{crit_{p},ord_{\ell}}$.

\section{Smoothness of eigencurve at critical weight $2$ Eisenstein series}
\label{sec:1}

Let us fix an integer $N\geq 1$, and a prime $p$ such that $p \nmid N$. We shall work with a subgroup $\Gamma$ of $SL_{2}(\Z)$ defined as $\Gamma = \Gamma_{0}(N) \cap \Gamma_{0}(p)$. For an integer $k \geq 2$, we shall denote by $M_{k}(\Gamma)$\;(resp. $S_{k}(\Gamma), M_{k}^{\dagger}(\Gamma),S_{k}^{\dagger}(\Gamma)$) the $\Q_{p}$-space of classical modular forms\;(resp. cuspidal classical modular forms, resp. overconvergent $p$-adic modular forms, resp. overconvergent cuspidal $p$-adic modular forms) of level $\Gamma$ and weight $k$. These spaces are acted upon by the Hecke operators $T_{l}$ for $l$ prime to $Np$, the Atkin-Lehner operator $U_{p}$ and Atkin-Lehner operators $U_{l}$ for $l \mid N$. Let $\HH$ and $\HH^{full}$ be commutative polynomial algebras over $\Z$ defined as,
$$ \HH = \Z[ (T_{l})_{l \nmid Np}, U_{p}] \text{ and } \HH^{full} = \Z[ (T_{l})_{l \nmid Np}, U_{p}, (U_{l})_{l \mid N}].$$
The spaces $M_{k}(\Gamma), S_{k}(\Gamma), M_{k}^{\dagger}(\Gamma),S_{k}^{\dagger}(\Gamma)$ are acted upon by both $\HH$ and $\HH^{full}$. Let $\C(N)$\;(resp. $\C^{full}(N)$, resp. $\C^{0,full}(N)$) denote the eigencurve (resp. full eigencurve, resp. full cuspidal eigencurve) of tame level $\Gamma_{0}(N)$ obtained via Buzzard's eigenvariety machine with Hecke algebra $\HH$\;(resp. $\HH^{full}$, resp. $\HH^{full}$) acting on the space of overconvergent modular forms\;(resp. overconvergent modular forms, resp. overconvergent cuspidal modular forms) of level $\Gamma$. Let $f$ be a normalized modular newform of weight $k$ and level $\Gamma_{0}(N)$. Let $T_{p}(f) = a_{p}f$ and let $\alpha$ and $\beta$ be two roots of the equation $x^{2}-a_{p}x+p^{k-1}=0$. Define $f_{\alpha}$ and $f_{\beta}$ as $f_{\alpha}(z)=f(z)-\beta f(pz)$ and $f_{\beta}(z)= f(z)-\alpha f(pz)$, then $f_{\alpha}$ and $f_{\beta}$ are normalized modular form of weight $k$ and level $\Gamma$, and they appear in the eigencurve $\C(N)$ and $\C^{full}(N)$. Similarly, modular forms in $M_{k}(\Gamma)$ which are also in $S_{k}^{\dagger}(\Gamma)$ appear in the eigencurve $\C^{0,full}(N)$. A modular form $f$ of weight $k$, which appears in the eigencurve $\C(N)$\;(resp. $C^{full}(N)$, resp. $C^{0,full}(N)$) can be viewed as a system of $\HH$\;(resp. $\HH^{full}$, resp. $\HH^{full}$)-eigenvalues appearing in $M_{k}^{\dagger}(\Gamma)$\;(resp. $M_{k}^{\dagger}(\Gamma)$, resp. $S_{k}^{\dagger}(\Gamma)$). From the construction of the eigencurve, we have, $\C^{full}(N) \twoheadrightarrow \C(N)$, $\C^{0,full}(N) \hookrightarrow \C^{full}(N)$ and if $N_{1}\mid N_{2}$ then $\C(N_{1}) \hookrightarrow \C(N_{2})$.\\

Now let us recall some well known facts about Eisenstein series of weight $2$ and its appearance in the eigencurve. Let $\chi$ and $\psi$ be two primitive Dirichlet characters with conductor $L$ and $R$ respectively and let $k$ be an integer such that $\chi(-1)\psi(-1)=(-1)^{k}$. Let $E_{k,\chi,\psi}(q)$ denotes the formal power series  
$$E_{k,\chi,\psi}(q) = c_{0}+ \sum_{m \geq 1} (\sum_{n \mid m} \psi(n)\chi(\frac{m}{n}) n^{k-1}) q^{m},$$
where, $ c_{0} = 0 $ if $L>1$ and $c_{0}= \frac{-B_{k,\psi}}{2k}$ if $L=1$, where $B_{k,\psi}$ is the generalized Bernouli number attached $\psi$ defined by the identity,
$$ \sum_{a=1}^{R} \frac{\psi(a)xe^{ax}}{e^{Rx}-1} = \sum_{k=0}^{\infty} B_{k,\psi} \frac{x^{k}}{k!}.$$
Except when $k=2$ and $\chi=\psi=1$, the power series $E_{k,\chi,\psi}(q^{t})$ defines an element of $M_{k}(RLt,\chi\psi)$. When $k=2$ and $\chi=\psi=1$, let us denote by $E_{2}(q) = E_{2,1,1}(q)$, then for $t>1$, $E_{2}(q)-tE_{2}(q^{t})$ is a modular form in $M_{2}(\Gamma_{0}(t))$\;(see \cite{Miyake},\cite{Stein}). Note that, for any prime $p$, the coefficient of $q^{p}$ in $E_{k,\chi,\psi}(q)$ is given by $a_{p} = \chi(p)+\psi(p)p^{k-1}$. Let $(k,\chi,\psi) \neq (2,1,1)$, that is, $E_{k,\chi,\psi}(q)$ is a modular form of weight $k$ and level $\Gamma_{1}(RL)$, then define,
$$ E_{k,\chi,\psi}^{ord_{p}}(q) := E_{k,\chi,\psi}(q)- \psi(p)p^{k-1}E_{k,\chi,\psi}(q^{p}) = E_{k,\chi,\psi}(z)- \psi(p)p^{k-1}E_{k,\chi,\psi}(pz) \text{ and }$$
$$ E_{k,\chi,\psi}^{crit_{p}}(q) := E_{k,\chi,\psi}(q)- \chi(p)E_{k,\chi,\psi}(q^{p}) = E_{k,\chi,\psi}(z)- \chi(p)E_{k,\chi,\psi}(pz).$$
$E_{k,\chi,\psi}^{ord_{p}}(q)$ and $E_{k,\chi,\psi}^{crit_{p}}$ are called ordinary and critical refinement at $p$ of $E_{k,\chi,\psi}(q)$, and are modular forms of weight $k$ level $\Gamma_{1}(RL)\cap \Gamma_{0}(p)$. In the case when $k=2$ and $\chi=\psi=1$, we can similarly define ordinary and critical refinement of the power series $E_{2}(q)$ by the formula,
$$E_{2}^{ord_{p}}(q)= E_{2}(q)-pE_{2}(q^{p}) \text{ and }$$
$$E_{2}^{crit_{p}}(q)= E_{2}(q)-E_{2}(q^{p}).$$
We see that $E_{2}^{ord_{p}}(q)$ is a modular form appearing in $M_{2}(\Gamma_{0}(p))$. It is apparent that $E_{2}^{crit_{p}}(q)$ is not a classical modular form since $E_{2}(q)$ is not a classical modular form and is a linear combination of $E_{2}^{ord_{p}}(q)$ and $E_{2}^{crit_{p}}(q)$. It is well known that for any prime $p$, $E_{2}(q)$ is $q$-expansion of a $p$-adic modular form of weight $2$ and level $1$\;(one can find explicit sequence of classical modular forms whose $q$ expansion converge to $E_{2}(q)$). It is natural to ask, whether $E_{2}(q)$ (and $E_{2}^{crit_{p}}(q)$) is an overconvergent $p$-adic modular form. For $p=2,3$ due to work of Koblitz \cite{Koblitz} and for $p \geq 5$ due to work of Coleman,Gouv{\^e}a and Jochnowitz \cite{Gouvea95}, it is known that $E_{2}(q)$(and $E_{2}^{crit_{p}}(q)$) is not an overconvergent $p$-adic modular form. This explains why $E_{2}^{crit_{p}}(q)$ does not appear in eigencurves $\C(1),\C^{full}(1)$ and $\C^{0,full}(1)$. Let $\ell \neq p$ be a prime, then $E_{2}^{ord_{\ell}}(q) := E_{2}(q)- \ell E_{2}(q^{\ell})$ is a modular form of weight $2$ and level $\Gamma_{0}(\ell)$. Critical refinement at $p$ of $E_{2}^{ord_{\ell}}(q)$, defined as, $E_{2}^{crit_{p},ord_{\ell}}(q):= E_{2}^{ord_{\ell}}(q)-E_{2}^{ord_{\ell}}(q^p)$ is a classical modular form of weight $2$ and level $\Gamma_{0}(\ell) \cap \Gamma_{0}(p)$. In fact, $E_{2}^{crit_{p},ord_{\ell}}(q)$ appears in $S_{2}^{\dagger}(\Gamma_{0}(\ell) \cap \Gamma_{0}(p))$ \cite[Cor 2.5]{Bellaiche.2012}. We remark that $E_{2}^{crit_{p},crit_{\ell}}(q)$ is not an overconvergent modular form, because otherwise, $E_{2}^{crit_{p}}(q)$, which is a linear combination of $E_{2}^{crit_{p},ord_{\ell}}(q)$ and $E_{2}^{crit_{p},crit_{\ell}}(q)$ would be a overconvergent modular form. This shows that $E_{2}^{crit_{p},ord_{\ell}}(q)$ appears as a point in the eigencurve $\C^{0,full}(\ell)$ and hence in $\C^{full}(\ell)$ for any prime $\ell \neq p$, which we will denote by $E_{2}^{crit_{p},ord_{\ell}}$. Since $E_{2}^{crit_{p},ord_{\ell}}$ appears in the eigencurve $\C^{0,full}(\ell)$, it is clear that any component containing $E_{2}^{crit_{p},ord_{\ell}}$ would have infinitely many cuspidal modular forms near it. In $\C^{0,full}(\ell)$ (and in $\C^{full}(\ell)$) $E_{2}^{crit_{p},ord_{\ell}}$ corresponds to the system of $\HH^{full}$-eigenvalues given by $\{ (a_{l^{\prime}} = 1+ l^{\prime})_{l^{\prime} \neq \ell,p}, a_{\ell} = 1, a_{p} = p \}$. Let us by abuse of notation, call  $E_{2}^{crit_{p}}$ the image of $E_{2}^{crit_{p},ord_{\ell}}$ under the map $\C^{full}(\ell)  \to \C(\ell)$. In $\C(\ell)$, $E_{2}^{crit_{p}}$ corresponds to the system of $\HH$-eigenvalues given by $\{ (a_{l^{\prime}} = 1+ l^{\prime})_{l^{\prime} \neq \ell,p}, a_{p} = p \}$. By analogy, in $\C(N)$, with $N>1$, we denote by $E_{2}^{crit_{p}}$ the system of $\HH$-eigenvalue given by $\{ (a_{l} = 1+l)_{l \nmid Np}, a_{p}=p \}$.\\

$\C(\ell)$ (or $\C^{full}(\ell)$, or $\C^{0,full}(\ell)$) naturally comes equipped with a universal pseudocharacter of dimension $2$, $\TT: G_{\Q,\ell p}  \to \OO(\C(\ell))$\;(or $\TT: G_{\Q,\ell p}  \to \OO(\C^{full}(\ell))$, or  $\TT: G_{\Q,\ell p}  \to \OO(\C^{0,full}(\ell))$ respectively). If $x \in \C(\ell)$\;(or in $\C^{full}(\ell)$, or in $\C^{0,full}(\ell)$) is a point, we denote the localization at $x$ by $\TT_{x} : G_{\Q,\ell p}  \to \OO_{x}$. Let $m$ be the maximal ideal of $A= \OO_{x}$, $k = A/m$ and $G = G_{\Q,\ell p}$. Let us assume that $x \in \C(\ell)$ is such that $\bar{\TT}_{x} := \TT \otimes A/m : k[G] \to k$ is sum of two distinct characters, $\bar{\TT}_{x} = \chi_{1} + \chi_{2}$. Then by \cite[Theorem 1.4.4]{Bellaiche.2009}, we have 
\begin{equation*}
R= A[G]/Ker \TT_{x} \cong \begin{bmatrix}A & B\\C & A\end{bmatrix}
\end{equation*}
where $B,C$ are two $A$-modules and we have an $A$-bilinear map $\phi: B \times C \to A$. Let us denote the image $\phi(B \times C)$ by $BC$; it is a proper ideal of $A$. $J=BC$ is the smallest ideal $I$ of $A$ such that $\TT_{A/I}:= \TT_{x} \otimes A/I$ is reducible\cite[Proposition 1.5.1]{Bellaiche.2009}. It is called the {\it ideal of total reducibility}. It follows from the proof of Proposition 1.5.1 in \cite{Bellaiche.2009} that $J \neq 0$ if $\TT_{x}$ is not sum of two characters whose reduction modulo $m$ are $\chi_{1}$ and $\chi_{2}$. Moreover we have,
\begin{prop}\cite[Theorem 1.5.5]{Bellaiche.2009}\label{B,C inj}
There exists injective natural maps of $k$-vector spaces
\begin{align*}
 i_{B}:(B/mB)^{\vee} &\hookrightarrow Ext_{G,cts}^{1}( \chi_{1},\chi_{2}) \\
 i_{C}:(C/mC)^{\vee} &\hookrightarrow Ext_{G,cts}^{1}(\chi_{2},\chi_{1})
\end{align*}
\end{prop}

An important ingredient in our proof will be a version of Kisin's lemma proved in \cite{Bellaiche.2006b}.
\begin{lem}\cite[Lemme 6]{Bellaiche.2006b}\label{kisin}
Let $I \subset m$ be any cofinite length ideal of $A=\OO_{x}$. Then\\
 $ D_{cris}(\rho_{|D} \otimes A/I)^{\phi= U_{p}}$ is free of rank 1 over $A/I$, here $D$ is the decomposition group at $p$.
\end{lem}

Last key ingredient for the proof of smoothness results is dimension formula for Selmer groups. Let $V$ be a $p$-adic representation of $G_{\Q}$ unramified almost everywhere. Let $\LL = (\LL_v)$ be a Selmer structure for $V$, that is, a family of subspaces $\LL_{v}$ of $H^{1}(G_{v},V)$ for all finite places $v$ of $\Q$, such that for all $v \notin \Sigma$ we have $\LL_{v} = H^{1}_{unr}(G_{v},V)$, where $\Sigma$ is a finite set of finite places of $\Q$ containing $p$. The Selmer group attached to the Selmer structure $\LL$ is defined as 
$$H^{1}_{\LL}(G_{\Q},V) := \mathrm{ker} (H^{1}(G_{\Q,\Sigma},V) \to \prod_{v \in \Sigma} H^{1}(G_{v},V)/ \LL_{v} ).$$
We note that, if for all $v \in \Sigma$, $\LL_{v} = H^{1}(G_{v},V)$, then $H^{1}_{\LL}(G_{\Q},V) = H^{1}(G_{\Q,\Sigma},V)$. If $\LL$ is a Selmer structure for $V$, then we define a Selmer structure $\LL^{\perp}$ for $V^{\ast}(1)$ by taking $\LL_{v}^{\perp}$ to be the orthogonal of $\LL_{v}$ in $H^{1}(G_{v},V^{\ast}(1))$. The following proposition, which is a consequence of Poitou-Tate machinery, allows us to relate dimension of Selmer group and its dual.
\begin{prop}\label{Selmerformula}
\begin{align*}
dim\text{ } H_{\LL}^{1}(G_{\Q},V) &= dim\text{ } H_{\LL^{\perp}}^{1}(G_{\Q},V^{\ast}(1))\\ 
&+ dim\text{ } H^{0}(G_{\Q},V) - dim\text{ } H^{0}(G_{\Q},V^{\ast}(1))\\
&+ \sum_{v \text{ places of } \Q}(dim\text{ } \LL_{v}- dim\text{ } H^{0}(G_{v},V)).
\end{align*}
\end{prop}
This formula\;(rather its analog for finite coefficients) is due to A.Wiles and can be found in \cite[Theorem 8.7.9]{NSW}, the version used here can be found in \cite[Proposition 2.7]{JoelBK}.\\

To show that the eigencurve $\C(\ell)$ is smooth at $x=E_{2}^{crit_{p}}$, suppose that $\mathrm{dim}_{k} (B/mB)^{\vee} = \mathrm{dim}_{k} (C/mC)^{\vee}  =1$. Then it was shown by Bella{\"i}che and Chenevier \cite[section 5.4]{Bellaiche.2006b}, that $J= BC$ is the maximal ideal $m$. The main ingredient of their proof was Kisin's lemma (\ref{kisin}).We will use this to prove the following theorem. Our proof closely follows the proof of Bellaiche and Chenevier as in \cite{Bellaiche.2006b}. We show that, in this situation, $(B/mB)^{\vee}$ actually injects into a selmer group, and we compute the dimension of the selmer group to conclude that $\mathrm{dim}_{k} (B/mB)^{\vee} = 1$.

\begin{thm}\label{E2smatl}
The eigencurve $\C(\ell)$, of tame level $\ell$ , is smooth at the critical Eisenstein series of weight 2, $E_{2}^{crit_{p}}$, where $\ell \neq p$ is a prime.
\end{thm}

\begin{proof}
Let $x$ be the point $E_{2}^{crit_{p}}$ in $\C(\ell)$. Let $A=\OO_{x}$ be the local ring at $x$ and let $m$ be the maximal ideal and $k=A/m$. We have a pseudocharacter $\TT_{x}:G_{\Q,lp} \to A$, such that $\bar{\TT_{x}}:= \TT_{x} \otimes A/m = 1 + \omega_{p}^{-1}$, where $\omega_{p}$ is the $p$-adic cyclotomic character, by proposition \ref{B,C inj} we have,
\begin{align*}
 i_{B}:(B/mB)^{\vee} &\hookrightarrow Ext_{G_{\Q,lp},cts}^{1}( \omega_{p}^{-1},1) \cong H^{1}(G_{\Q,lp},\omega_{p}) \\
 i_{C}:(C/mC)^{\vee} &\hookrightarrow Ext_{G_{\Q,lp},cts}^{1}(1,\omega_{p}^{-1}) \cong H^{1}(G_{\Q,lp},\omega_{p}^{-1})
\end{align*}
Let denote by $g(B), g(C)$ the dimension over $k$ of $(B/mB)^{\vee}, (C/mC)^{\vee}$ respectively. We note that $g(B)$ or $g(C)$ is $0$ would imply that $J=BC=0$, so in this situation $g(B),g(C)>0$.
Using Proposition \ref{Selmerformula} to $H^{1}(G_{\Q,lp},\omega_{p}^{-1})$ we see that it is $1$-dimensional and we have $g(C)=1$. Similarly, applying  Proposition \ref{Selmerformula} to $H^{1}(G_{\Q,lp},\omega_{p})$, we see that $\mathrm{dim} H^{1}(G_{\Q,lp},\omega_{p}) =2$, so we can not conclude that $g(B)=1$. But by applying Kisin's lemma (\ref{kisin}), we easily see that, $(B/mB)^{\vee}$ actually sits inside a subgroup of $H^{1}(G_{\Q,lp},\omega_{p})$, namely the Selmer group $H_{\LL}^{1}(G_{\Q},\omega_{p})$, where the local Selmer conditions $\LL_{v}$ are given by
\begin{equation*}
\LL_{v}=
\begin{cases}
H_{unr}^{1}(G_{v},\omega_{p}) & \text{if } v \neq l,p,\infty,\\
H_{f}^{1}(G_{\Q_{p}},\omega_{p}) & \text{if } v=p,\\
H^{1}(G_{\Q_{l}},\omega_{p}) & \text{if } v=l,\\
0 & \text{if } v= \infty.
\end{cases}
\end{equation*}
To see this, let $\rho$ be a representation of $G_{\Q,lp}$, which is an extension of $\omega_{p}^{-1}$ by $1$ and is in the image of $i_{B}$. We have an exact sequence 
\begin{equation*}
0 \to 1|_{G_{\Q_{p}}} \to \rho|_{G_{\Q_{p}}} \to \omega_{p}^{-1}|_{G_{\Q_{p}}} \to 0
\end{equation*}
By the left exactness of the functor $D_{cris}$, we see $D_{cris}(\rho|_{G_{\Q_{p}}})$ contains $D_{cris}(1|_{G_{\Q_{p}}})$, which is generated by a nonzero eigenvector of eigenvalue $1$. Applying Kisin's Lemma (\ref {kisin}) with $I=m$, tells us that $D_{cris}(\rho|_{G_{\Q_{p}}})$ contains a nonzero eigenvector of eigenvalue $p$. Thus $D_{cris}(\rho|_{G_{\Q_{p}}})$ has dimension $2$, in other words, $\rho$ is crystalline at $p$.\\

Now we compute the dimension of this Selmer group using the Proposition \ref{Selmerformula} as below:
\begin{align*}
dim\text{ } H_{\LL}^{1}(G_{\Q},\omega_{p}) &= dim\text{ } H_{\LL^{\perp}}^{1}(G_{\Q},\omega_{p}^{\ast}(1))\\ 
&+ dim\text{ } H^{0}(G_{\Q},\omega_{p}) - dim\text{ } H^{0}(G_{\Q},\omega_{p}^{\ast}(1))\\
&+ \sum_{v}(dim\text{ } \LL_{v}- dim\text{ } H^{0}(G_{v},\omega_{p})).
\end{align*}

Since $\omega_{p}^{\ast}(1) = \Q_p $, we have $dim\text{ } H_{\LL^{\perp}}^{1}(G_{\Q},\omega_{p}^{\ast}(1)) = 0$, as class group of $\Q$ is finite. Since $dim\text{ } H_{unr}^{1}(G_{v},\omega_{p}) = dim\text{ } H^{0}(G_{v},\omega_{p})$, there is no contribution coming from the local terms other than $ v = l,p,\infty$.
Let us compute the dimension of the remaining terms. From two global terms we have,\\
$ dim\text{ } H^{0}(G_{\Q}, \omega_{p})=0$, $dim\text{ } H^{0}(G_{\Q},\omega_{p}^{\ast}(1)) = dim\text{ } H^{0}(G_{\Q},1)= 1$.\\
At $v= l$, we have $$dim\text{ } H^{1}(G_{\Q_{l}},\omega_{p}) - dim\text{ } H^{0}(G_{\Q_{l}},\omega_{p}) = dim\text{ } H^{2}(G_{\Q_{l}},\omega_{p}) = dim\text{ } H^{0}(G_{\Q_{l}},\omega_{p}^{\ast}(1)) =1.$$\\
At $v=p$, we have $$dim\text{ } H_{f}^{1}(G_{\Q_{p}},\omega_{p}) - dim\text{ } H^{0}(G_{\Q_{p}},\omega_{p}) = \text{number of negative HT weights of }\omega_{p} =1.$$\\
At $v=\infty$, we have $ dim\text{ } \LL_{v} - dim\text{ } H^{0}(G_{\R},\omega_{p}) = 0$.\\
Thus putting this all together we get,
$$ dim\text{ } H_{\LL}^{1}(G_{\Q},\omega_{p}) = 0+ 0 -1 + 1 + 1 +0 =1.$$
Since $(B/mB)^{\vee} \hookrightarrow H_{\LL}^{1}(G_{\Q},\omega_{p})$ and since $ dim\text{ } H_{\LL}^{1}(G_{\Q},\omega_{p}) =1$, we get that $g(B)=1$. 

Since we have, $g(B)=g(C)=1$, by the work of Bella{\"i}che and Chenevier in \cite[section 5.4]{Bellaiche.2006b}, we get,  $BC=m$, the maximal ideal. We see that,
$$  1= g(B)g(C) \geq g(BC) =g(m) = dim_{k} \text{ } m/m^{2} \geq dim \text{ } A $$
Thus all the inequalities are actually equalities, and hence $A$ is regular local ring of dimension one. Thus $A$ is DVR and the eigencurve is smooth at $x$.
\end{proof}

\begin{cor}\label{E2fullsmatl}
The full eigencurve $\C^{full}(\ell)$ of tame level $\ell$ , is smooth at critical Eisenstein series of weight 2, $E_{2}^{crit_{p},ord_{l}}$, where $\ell \neq p$ is a prime. Hence the full cuspidal eigencurve $\C^{0,full}(\ell)$ of tame level $\ell$,  is smooth at  $E_{2}^{crit_{p},ord_{\ell}}$.
\end{cor}

\begin{proof}
Let $x$ be the point in $\C(\ell)$ corresponding to $E_{2}^{crit_{p}}$. Let $f: \C^{full}(\ell) \to \C(\ell)$ be the surjective map. Then $E_{2}^{crit_{p},ord_{l}}$ is the only preimage of $x$\;(since if $E_{2}^{crit_{p},crit_{l}}$ is not an overconvergent modular form). Almost all the classical points near $x$ in $\C(\ell)$ are cuspidal modular forms new at $\ell$. Indeed, if there are infinitely many cusp forms of old level $\ell$, then $(B/mB)^{\vee} \hookrightarrow H_{\LL}^{1}(G_{\Q},\omega_{p})$, where the local Selmer conditions $\LL_{v}$ are given by
\begin{equation*}
\LL_{v}=
\begin{cases}
H_{unr}^{1}(G_{v},\omega_{p}) & \text{if } v \neq p,\infty,\\
H_{f}^{1}(G_{\Q_{p}},\omega_{p}) & \text{if } v=p,\\
0 & \text{if } v= \infty.
\end{cases}
\end{equation*}
 We can compute the dimension of this Selmer group using Proposition \ref{Selmerformula} and see that it has dimension $0$, which is a contradiction. Any newform under the map $f$ has only 1 preimage since a newform at $\ell$ is uniquely determined by the eigenvalue of Hecke operators away from $\ell$. Hence $f$ is a local isomorphism at $x$. Since the point $x$ is smooth in $\C(\ell)$, the point $E_{2}^{crit_{p},ord_{l}}$ is smooth in $\C^{full}(\ell)$. Since $\C^{0,full}(\ell) \hookrightarrow \C^{full}(\ell)$ and both are equidimensional of dimension 1 and $E_{2}^{crit_{p},ord_{l}} \in \C^{0,full}(\ell)$, the eigencurve $\C^{0,full}(\ell)$ is smooth at $E_{2}^{crit_{p}, ord_{l}}$.
\end{proof}

{\bf Remark:} If $\ell_{1},\ell_{2} \neq p$ be two distinct primes, then $E_{2}^{crit_{p}}$ is non-smooth in $\C(\ell_{1}\ell_{2})$. $E_{2}^{crit_{p}}$ is smooth in both $\C(\ell_{1})$ and $\C(\ell_{2})$, since $\C(\ell_{1}) \hookrightarrow \C(\ell_{1}\ell_{2}) \hookleftarrow \C(\ell_{2})$, there exists a component of $\C(\ell_{1})$ and a component of $\C(\ell_{2})$ passing through $E_{2}^{crit_{p}}$ in $\C(\ell_{1}\ell_{2})$. In fact, extending this line of argument we can show that the classical modular  form $E_{2}^{crit_{p},ord_{l_{1}},ord_{l_{2}}}$ is non-smooth in the full eigencurve $\C^{full}(\ell_{1}\ell_{2})$. We also remark that our method does not provide information about smoothness of the eigencurve at $E_{2}^{crit_{p},ord_{l_{1}},crit_{l_{2}}}$ and $E_{2}^{crit_{p},crit_{l_{1}},ord_{l_{2}}}$.

\begin{cor}\label{nonsmooth}
The point corresponding to $E_{2}^{crit_{p},ord_{l_{1}},ord_{l_{2}}}$ is non-smooth in the full eigencurve $\C^{full}(\ell_{1}\ell_{2})$ of tame level $\ell_{1}\ell_{2}$. 
\end{cor}

\begin{proof}
Let $f$ be a classical modular newform of level $\Gamma_{0}(\ell_{1})$ appearing in the eigencurve $\C^{full}(\ell_{1})$. Let $f_{\alpha}$ and $f_{\beta}$ be two refinements at $\ell_{2}$ of $f$, then both $f_{\alpha}$ and $f_{\beta}$ appear in the eigencurve $\C^{full}(\ell_{1}\ell_{2})$. We see from the proof of Corollary \ref{E2fullsmatl} that $E_{2}^{crit_{p},ord_{l_{1}}}$ is smooth in the eigencurve $\C^{full}(\ell_{1})$ and almost all classical points near $E_{2}^{crit_{p},ord_{l_{1}}}$ are cuspidal newform of level $\ell_{1}$. Two refinements of $E_{2}^{crit_{p},ord_{l_{1}}}$, namely $E_{2}^{crit_{p},ord_{l_{1}},ord_{l_{2}}}$ and $E_{2}^{crit_{p},ord_{l_{1}},crit_{l_{2}}}$ appear in the eigencurve $\C^{full}(\ell_{1}\ell_{2})$. So there exists a component of $\C^{full}(\ell_{1}\ell_{2})$ passing through $E_{2}^{crit_{p},ord_{l_{1}},ord_{l_{2}}}$ which contains infinitely many cuspidal newforms of level $\ell_{1}$. Exactly same argument with $\C^{full}(\ell_{2})$ and $E_{2}^{crit_{p},ord_{l_{2}}}$ would show us that there exists a component of $\C^{full}(\ell_{1}\ell_{2})$ passing through $E_{2}^{crit_{p},ord_{l_{1}},ord_{l_{2}}}$ which contains infinitely many cuspidal newforms of level $\ell_{2}$. If these two components were same, then in the component at a Zariski dense set of points, the associated Galois representation $\rho_{f}$ would be unramified everywhere except $p$ and crystalline at $p$. Then we will have,$(B/mB)^{\vee} \hookrightarrow H_{\LL}^{1}(G_{\Q},\omega_{p})$, where the local Selmer conditions $\LL_{v}$ are given by
\begin{equation*}
\LL_{v}=
\begin{cases}
H_{unr}^{1}(G_{v},\omega_{p}) & \text{if } v \neq p,\infty,\\
H_{f}^{1}(G_{\Q_{p}},\omega_{p}) & \text{if } v=p,\\
0 & \text{if } v= \infty.
\end{cases}
\end{equation*}
As we saw in the proof of Corollary \ref{E2fullsmatl}, the dimension of this Selmer group is $0$, which is a contradiction. Hence these two components are distinct and $E_{2}^{crit_{p},ord_{l_{1}},ord_{l_{2}}}$ is non-smooth in the eigencurve $\C^{full}(\ell_{1}\ell_{2})$.
\end{proof}

Let $\chi$ is a Dirichlet character of conductor $\ell$, where $\ell \neq p$ is a prime. We will show that the eigencurve $\C(\ell^{2})$ is smooth at the point corresponding to the generalized Eisenstein series $E_{2,\chi,\chi^{-1}}^{crit_{p}}$. Our proof is exactly same as of Theorem \ref{E2smatl}, we just need to compute dimension of slightly different Selmer groups.

\begin{prop}\label{smatl1}
The eigencurve $\C(\ell^2)$ is smooth at the point corresponding to $E_{2,\chi,\chi^{-1}}^{crit_{p}}$.
\end{prop}

\begin{proof}
Let $y$ be the point in $\C(\ell^2)$ corresponding to $E_{2,\chi,\chi^{-1}}^{crit_{p}}$. Let $A=\OO_{y}$ be the local ring at $y$ and $m$ be the maximal ideal. Then we have pseudocharacter $\TT_{y}:G_{\Q,lp} \to A$, such that $\bar{\TT_{y}}:= \TT_{y} \otimes A/m = \chi +\chi^{-1} \omega_{p}^{-1}$, here $\omega_{p}$ is the $p$-adic cyclotomic character.
By construction and proposition(\ref{B,C inj}) we have 
\begin{align*}
 i_{B}:(B/mB)^{\vee} &\hookrightarrow Ext_{G_{\Q,lp},cts}^{1}(\chi^{-1}\omega_{p}^{-1},\chi) \cong H^{1}(G_{\Q,lp},\chi^{2}\omega_{p}) \\
 i_{C}:(C/mC)^{\vee} &\hookrightarrow Ext_{G_{\Q,lp},cts}^{1}(\chi , \chi^{-1}\omega_{p}^{-1}) \cong H^{1}(G_{\Q,lp},\chi^{-2}\omega_{p}^{-1})
\end{align*}
A similar argument as in Theorem \ref{E2smatl} shows that $g(C)=1$ and $(B/mB)^{\vee}$ sits inside the Selmer group $H_{\LL}^{1}(G_{\Q},\chi^{2}\omega_{p})$, where the local Selmer conditions $\LL_{v}$ are given by
\begin{equation*}
\LL_{v}=
\begin{cases}
H_{unr}^{1}(G_{v},\chi^{2}\omega_{p}) & \text{if } v \neq l,p,\infty,\\
H_{f}^{1}(G_{\Q_{p}},\chi^{2}\omega_{p}) & \text{if } v=p,\\
H^{1}(G_{\Q_{l}},\chi^{2}\omega_{p}) & \text{if } v=l,\\
0 & \text{if } v= \infty.
\end{cases}
\end{equation*}

Assume that $\chi \neq \chi^{-1}$, note that the first term in the right hand side of dimension formula for the Selmer group in Proposition \ref{Selmerformula} vanishes by the work of Soul\'{e} on Bloch-Kato conjecture\cite{MR618313} as $\chi^{2}\omega_{p}$ is pure of motivic weight -2. We can compute the remaining terms and see that $g(B)=1$.\\
Now assume that $\chi = \chi^{-1}$, then the Selmer group $H_{\LL}^{1}(G_{\Q},\chi^{2}\omega_{p})$ is same as that of Theorem \ref{E2smatl}, and we see that $g(B)=1$.
\end{proof}

\section{$E_{2}^{crit_{p}}$ is \'etale over the weight space in $\C(\ell)$}

In the previous section we proved that the eigencurve $\C(\ell)$ is smooth at $E_{2}^{crit_{p}}$. In this section we will show that the eigencurve is \'etale over the weight space at this point.\\

Let $A$ be a discrete valuation domain and $K$ be the field of fraction and $k$ be the residue field of $A$. Let $\pi$ denote a uniformizer of $A$. Let $(\rho,V)$ be a representation of $G$ of dimension $2$. Let $X$ denote the set of stable lattices of $V$ upto homothety. Let $x$ and $x^{\prime}$ be two points in $X$, then $x$ and $x^{\prime}$ are neighbor if there exists lattices $\Lambda$, $\Lambda^{\prime}$ in $V$, such that $x = [\Lambda]$, $x^{\prime}= [\Lambda^{\prime}]$ and $ \pi \Lambda \subset \Lambda^{\prime} \subset \Lambda$. The set $X$ is a tree, called the Bruhat-Tits tree of $GL_{2}(K)$. Let $S$ denote the subset of $X$ fixed by $\rho(G)$. Then $S$ is nonempty and bounded. If $x \in S$ and $x= [\Lambda] = [\Lambda^{\prime}]$, then $\bar{\rho}_{\Lambda} = \bar{\rho}_{\Lambda^{\prime}}$, hence there is no ambiguity in calling that representation $\bar{\rho}_{x}$. If $x$ is in $S$, then $x$ has no neighbor in $S$ if and only if $\bar{\rho}_{x}$ is irreducible, $x$ has exactly one neighbor in $S$ if and only if $\bar{\rho}_{x}$ is reducible but indecomposable and $x$ has more than one neighbor in $S$ if and only if $\bar{\rho}_{x}$ is sum of two characters. The number of neighbors is 2 if the two character appearing in $\bar{\rho}_{x}$ are distinct. Assume that $\bar{\rho}^{ss}$ be sum of two distinct characters $\chi_{1},\chi_{2}: G \to k^{*}$. Let $l$ be the length of the set $S$\;(since $\bar{\rho}^{ss}$ is sum of two distinct characters, number of neighbors is at most 2, hence $S$ is a segment $[x_{0},x_{l}]$). Let $n= n(\rho)$ be the largest integer such that there exists two characters $\psi_{1}, \psi_{2}: G \to (A/\pi^{n}A)^{*}$, such that for all $g \in G$, $tr(\rho(g)) = \psi_{1}(g) + \psi_{2}(g)$ (mod $\pi^{n}$). Then $l =n$\;(for details see \cite{MR3201061}). \\

With these notations we prove:

\begin{thm}\label{E2etale}
The eigencurve $\C(\ell)$ is \'etale over the weight space at the point $x = E_{2}^{crit_{p}}$ .
\end{thm}

\begin{proof}
Let $A = \OO_{x}$. Since the eigencurve is smooth at $x$, $A$ is a discrete valuation domain. The ideal of reducibility of $\TT_{x}: G_{\Q,lp} \to A$, is $ J= m = (\pi)$. So the length of the segment $S$ is 1. By the work of Bellaiche and Chenevier it is known that the ideal of irreducibility of $\TT|_{G_{\Q_{p}}}:G_{\Q_{p}} \to A$, is $J_{p} = (\kappa - \kappa(x))$ \cite[Theorem 2]{Bellaiche.2006b}, where $\kappa: \C(\ell) \to \W$ is the weight map. Thus, to prove that the eigencurve is \'etale over the weight space at $E_{2}^{crit_{p}}$ it is enough to show $J_{p}$ is also the maximal ideal. $A$ is a discrete valuation domain, so $J_{p} = (\pi^{n})$ for some $n$. A priori,  $S_{p}$ has length $n$\;($S_p$ is the set described as above for $\rho|_{G_{\Q_{p}}}$), also $S \subset S_{p}$. We want to prove $S_{p}$ has length $1$.\\

 $S$ has length 1, so it consists of 2 points, say $x_{0}$ and $x_{1}$. Since both these points have exactly 1 neighbor, $\bar{\rho}_{x_{i}}$ is reducible but indecomposable for $i=0,1$. Say $\bar{\rho}_{x_{0}} \in Ext_{G_{\Q,lp}}(\omega_{p}^{-1},1)$, then $\bar{\rho}_{x_{1}} \in Ext_{G_{\Q,lp}}(1,\omega_{p}^{-1})$. In fact as in the proof of \ref{E2smatl},  $\bar{\rho}_{x_{0}}$ belongs to a Selmer group $H_{\LL}^{1}(G_{\Q},\omega_{p})$, where the local Selmer conditions $\LL_{v}$ are given by
\begin{equation*}
\LL_{v}=
\begin{cases}
H_{unr}^{1}(G_{v},\omega_{p}) & \text{if } v \neq l,p,\infty,\\
H_{f}^{1}(G_{\Q_{p}},\omega_{p}) & \text{if } v=p,\\
H^{1}(G_{\Q_{l}},\omega_{p}) & \text{if } v=l,\\
0 & \text{if } v= \infty.
\end{cases}
\end{equation*}
 
We have a map $ res: H_{\LL}^{1}(G_{\Q},\omega_{p}) \to H^{1}(G_{\Q_{p}}, \omega_{p})$. Since $H_{\LL}^{1}(G_{\Q},\omega_{p}) \cong \OO_{\Q,\{ l \}}^{*} \otimes_{\Z} \Q_{p}$, it is generated by $l$. On the other hand, $H^{1}(G_{\Q_{p}}, \omega_{p}) \cong \widehat{\Q^{*}_{p}} \otimes_{\Z_{p}} \Q_{p}$. Under the restriction map $l$ goes to $l$. So the map $ res: H_{\LL}^{1}(G_{\Q},\omega_{p}) \hookrightarrow H^{1}(G_{\Q_{p}}, \omega_{p})$ is injective, hence the element corresponding to $\bar{\rho}_{x_{0}}$ under the restriction maps to a nonzero element of $H^{1}(G_{\Q_{p}}, \omega_{p})$. Thus $\bar{\rho}_{x_{0}|G_{\Q_{p}}}$ is reducible but indecomposable as well, hence it has exactly one neighbor in $S_{p}$. \\

Suppose $S_{p}$ consists of the segment $[y_{0}, y_{n}]$, and the path connecting $y_0$ to $y_n$ is $y_{0}, y_{1}, \dots, y_{n}$. The point corresponding to $\bar{\rho}_{x_{0}|G_{\Q_{p}}}$ has exactly one neighbor, so it is one of the end points. Without loss of generality, let $y_{0}$ be the point corresponding to $\bar{\rho}_{x_{0}|G_{\Q_{p}}}$.\\

Now we consider $\bar{\rho}_{x_{1}} \in Ext_{G_{\Q,lp}}( 1, \omega_{p}^{-1}) \cong H^{1}(G_{\Q,lp}, \omega_{p}^{-1})$.\\ 
Consider the Selmer group $H_{\LL}^{1}(G_{\Q},\omega_{p}^{-1})$, where the local Selmer conditions $\LL_{v}$ are given by
\begin{equation*}
\LL_{v}=
\begin{cases}
H_{unr}^{1}(G_{v},\omega_{p}^{-1}) & \text{if } v \neq l,p,\infty,\\
H_{f}^{1}(G_{\Q_{p}},\omega_{p}^{-1}) & \text{if } v=p,\\
H^{1}(G_{\Q_{l}},\omega_{p}^{-1}) & \text{if } v=l,\\
0 & \text{if } v= \infty.
\end{cases}
\end{equation*}
Note that $\LL_{p}=H_{f}^{1}(G_{\Q_{p}},\omega_{p}^{-1}) = 0 $. \\
We also have an exact sequence from the definition of Selmer group as:
\begin{equation*}
0 \to H_{\LL}^{1}(G_{\Q},\omega_{p}^{-1}) \to H^{1}(G_{\Q,lp}, \omega_{p}^{-1}) \to H^{1}(G_{\Q_{p}},\omega_{p}^{-1})/ \LL_{p} \times H^{1}(G_{\Q_{l}},\omega_{p}^{-1})/ \LL_{l} \times H^{1}(G_{\R}, \omega_{p}^{-1})
\end{equation*}
Note that $H^{1}(G_{\Q_{p}},\omega_{p}^{-1})/ \LL_{p} \times H^{1}(G_{\Q_{l}},\omega_{p}^{-1})/ \LL_{l} \times H^{1}(G_{\R}, \omega_{p}^{-1}) \cong H^{1}(G_{\Q_{p}},\omega_{p}^{-1})$. Thus we have an exact sequence:
\begin{equation*}
0 \to H_{\LL}^{1}(G_{\Q},\omega_{p}^{-1}) \to H^{1}(G_{\Q,lp}, \omega_{p}^{-1}) \to H^{1}(G_{\Q_{p}},\omega_{p}^{-1})
\end{equation*}
Since $\mathrm{dim} H_{\LL^{\perp}}^{1}(G_{\Q},\omega_{p}^{2}) = 0$ by the work of Soul\'e \cite{MR618313} on Bloch-Kato conjecture, we can use the dimension formula for Selmer group as in Theorem \ref{E2smatl} and get that $\mathrm{dim} H_{\LL}^{1}(G_{\Q},\omega_{p}^{-1}) = 0$ as the contribution from both global and local terms are zero. Hence we have an injection $f: H^{1}(G_{\Q,lp}, \omega_{p}^{-1}) \hookrightarrow H^{1}(G_{\Q_{p}},\omega_{p}^{-1})$. The element corresponding to $\bar{\rho}_{x_{1}}$ under the  map $f$ maps to a nonzero element of $H^{1}(G_{\Q_{p}}, \omega_{p}^{-1})$, hence it is reducible but indecomposable as well, hence it has exactly one neighbor in $S_{p}$. Thus it corresponds to the point $y_{n}$. \\

Since $S \subset S_{p}$, and since $x_{0} \mapsto y_{0}$ and $x_{1} \mapsto y_{n}$, we have $n =1$. Thus $J_{p} = (\pi) = m$.
\end{proof}

\begin{cor}\label{2rep}
There exists two representations $\rho_{1}, \rho_{2}: G_{\Q,{lp}} \to GL_{2}(A)$, such that $tr(\bar{\rho_{i}}) = 1+ \omega_{p}^{-1}$ for $i =1,2$, and $\bar{\rho}_{1}$ is a nontrivial extension of trivial representation by $\omega_{p}^{-1}$ and $\bar{\rho}_{2}$ is nontrivial extension of $\omega_{p}^{-1}$ by trivial representation. Moreover $(\bar{\rho}_{i})_{|_{G_{\Q_{p}}}}$ is reducible but indecomposable for $i=1,2$. 
\end{cor}

\begin{proof}
It is just a restatement of the fact the set $S$ and $S_{p}$ has two points.
\end{proof}

Let $\mathcal{T} : G_{\Q,lp}  \to \OO(\C^{0,full}(l))$ denote the universal pseudocharacter. If $\kappa \in \C^{0,full}(l)(\CC_{p})$, we denote the localization to $\kappa$ by $\mathcal{T}_{\kappa}$. A theorem of R. Taylor ensures that $\mathcal{T}_{\kappa}$ is trace of a continuous semisimple 2 dimensional representation $\rho_{\kappa}$ over $\CC_{p}$. Let $R$ denote the absolutely reducible locus in $\C^{0,full}(l)(\CC_{p})$ in the sense that $\rho_{\kappa}$ is absolutely reducible. Notice that $E_{2}^{crit_{p},ord_{l}} \in R$. Let $R^{0} = \C(l) \backslash R$. \\

To any $x \in \C^{0,full}(l)$ one can associate a 2 dimensional Weil-Deligne representation $(\rho_{x,l}, N_{x})$(see \cite{Bellaiche.2009} for further details). If moreover $x \in R^{0}$, the associated Weil-Deligne representation is unique(for details we refer reader to \cite{MR2827001}). Paulin made the following conjecture regarding geometry of the full cuspidal eigencurve.

\begin{conj}{\bf Level Lowering} \cite{Paulin.2012}(Paulin)
Let $Z \subset \C^{0,full}(l)$ be a generically special (at $l$) component such that there exists an $x \in Z$ such that $N_{x} = 0$ , then there exists $Z^{\prime} \subset \C^{0,full}(l)$, generically principal series (at $l$) irreducible component, such that $x \in Z^{\prime}$.
\end{conj}

Let $x = E_{2}^{crit_{p},ord_{l}}$, then by Corollary \ref{E2fullsmatl}, $\C^{0,full}(l)$ is smooth at $x$. Moreover all but finitely many classical points in the component containing $x$ are cusp forms new at $l$. So the component is generically special at $l$. Let $A = \OO_{x}$. By corollary \ref{2rep}, there are two representations $\rho_{1}, \rho_{2}: G_{\Q,{lp}} \to GL_{2}(A)$, such that $\bar{\rho_{i}}^{ss} = 1+ \omega_{p}^{-1}$ for $i =1,2$. Hence we can associated two WD representations to $x$, one for $\rho_{1}$ and one from $\rho_{2}$. We remark that this happens because $x \in R$.\\

 If $\rho_{1} \in Ext(1, \omega_{p}^{-1})$, then $\rho_{1|G_{\Q_{l}}} = 1 \oplus \omega_{p}^{-1}$, as $H^{1}(G_{\Q_{l}}, \omega_{p}^{-1}) = 0$. Hence $N_{x} = 0$ in this case. So even if the point $x=E_{2}^{crit_{p},ord_{l}}$ lies in a generically special component (at $l$) and $N_{x} = 0$, the point is smooth in the eigencurve. This is in contrast with Paulin's conjecture, which predicts the point $x$ will be non-smooth in the eigencurve. \\

But if we look at the other representation, $\rho_{2} \in Ext(\omega_{p}^{-1},1)$, then $\rho_{2|G_{\Q_{l}}}$ is a nontrivial extension of  $\omega_{p}^{-1}$ by the trivial character, hence $N_{x} \neq 0$. So probably the correct way to formulate the conjecture would be $N_{x} = 0$ for all the associated Weil-Deligne representaion to $x$. In light of Paullin's conjecture, we ask the following question: Let $Z \subset \C^{0,full}(N)$ be a generically special (at $l$) component such that there exists an $x \in Z$ such that $N_{x} = 0$ for all possible Weil-Deligne representation associated to $x$ at $l$, then does there exists $Z^{\prime} \subset \C^{0,full}(N)$, generically principal series (at $l$) irreducible component, such that $x \in Z^{\prime}$?

\subsection*{Acknowledgement}
I would like to thank Jo\"{e}l Bella\"\i che for suggesting the problem and a lot of helpful comments and ideas. Let me take this opportunity to thank him for his time and beautiful insights of the subject that he gave me. I would also like to thank Bala Sury for his comments and suggestions. I would also like to thank the referee for carefully reviewing my work and his valuable comments and suggestions. His comments have helped me in improving the manuscript. This work was done by the author as part of his PhD thesis in Brandeis University.

\bibliographystyle{amsplain}
\bibliography{e2crit}

\end{document}